\documentclass[12pt]{amsart} 

\usepackage[utf8]{inputenc}
\usepackage[T1]{fontenc}	
\usepackage[french,english]{babel}

\usepackage{lmodern}			
\usepackage{newtxtext}

\usepackage[top=3cm, bottom=3cm, left=3.6cm, right=3.6cm]{geometry}

\usepackage{graphicx}	

\usepackage{tikz}
\usetikzlibrary{patterns}

\usepackage[a4paper, plainpages=false, colorlinks, linkcolor=bleuFonce, citecolor=rougeFonce, urlcolor=vertFonce, breaklinks]{hyperref}
\usepackage{orcidlink}

\usepackage{placeins}
\usepackage{float} 

\usepackage{color}
\definecolor{vertFonce}	{rgb}{0,0.5,0}
\definecolor{numLignes}	{rgb}{0.17,0.57,0.7}	
\definecolor{gris}		{rgb}{0.5,0.5,0.5}
\definecolor{grisFonce}	{rgb}{0.2,0.2,0.2}
\definecolor{orange}	{rgb}{1,0.65,0.31}		
\definecolor{orangeFonce}{rgb}{1,0.4,0}
\definecolor{bleuFonce}	{rgb}{0,0,0.4}
\definecolor{rougeFonce}{rgb}{0.3,0,0}
\definecolor{rougeWord}	{rgb}{0.5,0,0}
\definecolor{vertClair}	{rgb}{0.8,1,0.8}
\definecolor{rougeClair}{rgb}{1,0.5,0.5}
\definecolor{violet}	{rgb}{0.5,0,0.5}

\usepackage{pict2e}
\usepackage{multido}

\usepackage{amsfonts,amssymb,amsthm,amsmath}
\usepackage{amsaddr}		
\usepackage{dsfont}			
\usepackage{mathrsfs}
\usepackage{yfonts}
\usepackage{cancel}
 
\newtheorem{lem}{Lemma}[section]
\newtheorem{theorem}{Theorem}[section]

\newtheorem{prop}{Proposition}[section]

\newtheorem{remark}{Remark}[section]

\newenvironment{system*}{%
	\begin{equation*}\left\{\ \begin{aligned}
}{%
	\end{aligned} \right. \end{equation*}%
}

\newcommand		{\N}		{\mathbb N}			

\newcommand		{\RR}		{\mathbb R}			
\newcommand		{\R}		{\RR}
\newcommand		{\Rd}		{\R^3}
\newcommand		{\Rdd}		{\R^6}
\newcommand		{\hd}		{h^3}
\newcommand		{\cP}		{\mathcal P}		
\newcommand		{\opP}		{\cP}				
\renewcommand	{\L}		{\mathcal L}		

\renewcommand	{\P}		{\mathscr P}		

\newcommand		{\lt}			{\left}				%
\newcommand		{\rt}			{\right}			%
\renewcommand	{\(}			{\lt(}
\renewcommand	{\)}			{\rt)}
\newcommand		{\set}[1]		{\lt\{#1\rt\}}
\newcommand		{\bangle}[1]	{\lt\langle #1\rt\rangle}
\newcommand		{\weight}[1]	{\bangle{#1}}	

\newcommand		{\com}[1]		{\lt[{#1}\rt]}		

\newcommand		{\n}[1]			{\lt|{#1}\rt|}
\newcommand		{\nrm}[1]		{\lt\|{#1}\rt\|}
\newcommand		{\snrm}[1]		{\lVert #1\rVert}

\newcommand		{\Nrm}[2]		{\nrm{#1}_{#2}}
\newcommand		{\sNrm}[2]		{\snrm{#1}_{#2}}

\newcommand		{\indic}	{\mathds{1}}		

\renewcommand		{\d}		{\mathop{}\!\mathrm{d}}		
\newcommand			{\dpt}		{\partial_t}

\newcommand			{\dt}		{\frac{\d}{\d t}}	

\newcommand			{\Dx}		{\nabla_x}
\newcommand			{\Dv}		{\nabla_\xi}

\DeclareMathOperator{\cF}		{\mathcal{F}}		
\DeclareMathOperator{\tr}		{Tr}				

\DeclareMathOperator*{\supess}	{sup\,ess}

\newcommand		{\F}[1]			{\cF\!\( #1 \)}		
\newcommand		{\Tr}[1]		{\tr\!\( #1 \)}		

\newcommand		{\intd}			{\int_{\Rd}}
\newcommand		{\intdd}		{\int_{\Rdd}}
\newcommand		{\iintd}		{\iint_{\Rdd}}

\newcommand		{\init}			{\mathrm{in}}
\newcommand		{\loc}			{\mathrm{loc}}

\newcommand		{\eps}			{\varepsilon}

\newcommand		{\cC}			{\mathcal{C}}

\usepackage{braket}

\newcommand		{\op}		{\boldsymbol{\rho}}	

\newcommand		{\opgam}	{\boldsymbol{\gamma}}
\newcommand		{\opc}		{\boldsymbol{c}}

\newcommand		{\Wh}		{\mathrm{W}_{2,\hbar}}		

\newcommand		{\opp}		{\boldsymbol{p}}

\title[\textsc{Enhanced Stability in Quantum Optimal Transport}]{\Large Enhanced Stability in Quantum Optimal Transport Pseudometrics: From Hartree to Vlasov--Poisson}

\author[\textsc{M.~Iacobelli}]{\large\textsc{Mikaela Iacobelli}}
\address[M.~Iacobelli]{ETH Z\"urich, Department of Mathematics,\\ R\"amistrasse 101, 8092 Z\"urich, Switzerland}
\email[M.~Iacobelli]{mikaela.iacobelli@math.ethz.ch}

\author[\textsc{L.~Lafleche}]{\large\textsc{Laurent Lafleche}}
\address[L.~Lafleche]{Unit\'e de Math\'ematiques pures et appliqu\'ees, \\\'Ecole Normale Supérieure de Lyon, Lyon, France}
\email[L.~Lafleche]{laurent.lafleche@ens-lyon.fr}


\begin{document}

\begin{abstract}
	In this paper we establish almost-optimal stability estimates in quantum optimal transport pseudometrics for the semiclassical limit of the Hartree dynamics to the Vlasov--Poisson equation, in the regime where the solutions have bounded densities. We combine Golse and Paul's method from~[Arch. Ration. Mech. Anal. 223:57–94, 2017], which uses a semiclassical version of the optimal transport distance and which was adapted to the case of the Coulomb and gravitational interactions by the second author in~[J. Stat. Phys. 177:20--60, 2019], with a new approach developed by the first author in~[Arch. Ration. Mech. Anal. 244:27--50, 2022] to quantitatively improve stability estimates in kinetic theory.
\end{abstract}

\begingroup
\def\uppercasenonmath#1{} 
\let\MakeUppercase\relax 
\maketitle
\thispagestyle{empty} 
\endgroup

\bigskip


 
\section{\large Introduction}

\subsection{General overview}

	In this article, we study the nonrelativistic quantum and classical equations that govern the dynamics of an infinite number of particles' density in the mean field regime. These equations are known as the Vlasov equation and the Hartree equation, respectively. More precisely we aim to give an almost-optimal convergence rate, in the semiclassical limit, from the Hartree equation to the Vlasov--Poisson equation, in the so-called quantum optimal transport pseudometrics.

\subsection{The Vlasov equation}

	The Vlasov equation is a non-linear partial differential equation that provides a statistical description of the collective behavior of a large number of charged particles undergoing mutual long-range interactions.

	This model was originally introduced by Jeans in the context of Newtonian stellar dynamics~\cite{jeans_theory_1915} and was later developed by Vlasov in his work on plasma physics~\cite{vlasov_vibrational_1938}. More precisely, it describes the evolution of the phase space distribution of particles $f = f(t,x,\xi)$ where $(x,\xi)\in \Rdd$ are the position and velocity variables and $t\in\R$ is the time. The equation then reads
	\begin{equation}\label{eq:Vlasov}
		\partial_t f + \xi\cdot\Dx f + E_f\cdot\Dv f = 0,
	\end{equation}
	where $E_f = -\nabla V_f$ is the force field corresponding to the mean-field potential
	\begin{equation*}
		V_f(x) = (K * \rho_f)(x) = \intd K(x-y)\,\rho_f(y)\d y
	\end{equation*}
	and $\rho_f$ is the position density of particles defined by
	\begin{equation}\label{eq:spatial_density}
		\rho_f(x) = \intd f(x,\xi)\d \xi.
	\end{equation}
	
	In the sequel we will be interested in the Vlasov--Poisson equation in dimension $d=3$. This equation serves as the classical kinetic model that characterizes dilute, completely ionized, and non-magnetized plasma, or galaxies. It corresponds to considering the coupling of the Liouville equation with the potential given by
	\begin{equation}\label{eq:Coulomb}
		K(x) = \frac{\kappa}{4\pi \n{x}}
	\end{equation}
	for some constant $\kappa\in\R$. Assuming that the unit is chosen so that particles have charge~$1$, in the case of the Coulomb law, the constant $\kappa$ is given by $\kappa = 1/\eps_0$ where $\eps_0$ is the electric constant. If the particles have mass~$1$, in the case of Newton's gravitational law, $\kappa$ is given by $\kappa = -4\pi\,G$ with $G$ the gravitational constant.

	The well-posedness theory for this system has undergone extensive investigation, and readers can refer to the comprehensive survey paper~\cite{griffin-pickering_recent_2021} for an overview. Notably, global-in-time classical solutions have been established under a range of initial data conditions, as explored in works such as~\cite{ukai_classical_1978, bardos_global_1985, batt_global_1991, lions_propagation_1991, pfaffelmoser_global_1992, schaeffer_global_1991}. Meanwhile, global-in-time weak solutions have been introduced in~\cite{arsenev_existence_1975} and \cite{horst_weak_1984}, primarily focusing on $L^p$ initial data, as elaborated in~\cite{bardos_global_1985, bardos_priori_1986}. Concerning uniqueness, we will focus on a very important contribution to the theory made by Loeper~\cite{loeper_uniqueness_2006}, who proved uniqueness for solutions of the Vlasov--Poisson equation with bounded density by means of a strong-strong stability estimate in the $2$-Wasserstein--(Monge--Kantorovich) distance. Loeper's result is the cornerstone for several stability arguments, and it is a key tool in proving the validity of the quasi-neutral limit for the Vlasov--Poisson equation~\cite{griffin-pickering_mean_2018, griffin-pickering_singular_2020, han-kwan_quasineutral_2017, han-kwan_quasineutral_2017, han-kwan_quasineutral_2017-1, griffin-pickering_stability_2023}. In this paper we will rely on a recent improvement of Loeper's estimate obtained in~\cite{iacobelli_new_2022} (see also \cite{iacobelli_stability_2022}).

\subsection{The Hartree equation}

	In the context of quantum mechanics, the distributions of particles is described by wave functions, and more generally, by operators, which can be seen as incoherent superposition of wave functions. More precisely, we define the set of density operators by
	\begin{equation*}
		\opP = \set{\op\in \L(L^2(\Rd)): \op\geq 0 \text{ and } \hd\Tr{\op} = 1},
	\end{equation*}
	where $\L(L^2(\Rd))$ denotes the set of linear operators acting on $L^2(\Rd)$, $\tr$ denotes the trace and $h$ is the Planck constant. Such operators are trace-class, and so automatically compact operators. For such operators, one can define a quantum version of the position density
	\begin{equation}\label{eq:spatial_density_quantum}
		\rho(x) = \hd\, \op(x,x),
	\end{equation}
	where the formula should in general be understood in the weak sense $\intd \rho\,\varphi = \hd \Tr{\op\,\varphi}$ for any $\varphi\in L^\infty(\Rd)$.
	
	The quantum analogue of the Vlasov--Poisson equation is the Hartree equation with Coulomb or gravitational potential, which can be written in its operator form as follows
	\begin{align}\label{eq:Hartree}
		i\hbar\,\partial_t \op &= \com{H_{\op},\op}, & H_{\op} &= -\frac{\hbar^2}{2}\, \Delta + V_{\op},
	\end{align}
	where $\hbar = \tfrac{h}{2\pi}$ is the reduced Planck constant, $\com{A, B} := AB-BA$ denotes the commutator of operators and $V_{\op}$ is the mean-field potential $V_{\op} = K*\rho(x)$ associated to $\rho(x)$, with $K$ again given by Equation~\eqref{eq:Coulomb}. By the spectral theorem, it can also be written as a system of infinitely many coupled Schr\"odinger equations (see e.g.~\cite{lions_sur_1993}).
	
	The well-posedness for this equation is also well understood. Global existence and uniqueness of solutions was proved in~\cite{brezzi_three-dimensional_1991} for repulsive interactions and in~\cite{illner_global_1994} for attractive interactions, for solutions in $H^2$, that is for operators such that $\Tr{-\Delta\op}$ is bounded. The regularity assumptions were then relaxed in~\cite{castella_$l^2$_1997} to the $L^2$ case, that is the case of trace-class density operators. In~\cite[Theorem~8.1]{castella_$l^2$_1997}, the propagation of higher Sobolev norms of the form $\Tr{(-\Delta)^n\op}$ for any $n\in\N$ is also proved, which can be interpreted as the quantum analogue of the propagation of moments.
	
	In the context of this paper, we will consider that the quantum position density is initially bounded uniformly in $\hbar$. This was proved to remain true for finite times in~\cite{lafleche_propagation_2019} by first proving the propagation of moments of the form $h^d\Tr{(-\hbar\Delta)^n\op}$ for any $n\in\N$ uniformly in $\hbar$. The global-in-time uniform-in-$\hbar$ propagation of these quantities for $n > 2$ in the case of the Coulomb or gravitational potentials remains an open problem. This was however obtained in the case of more decaying potentials and small data in~\cite{lafleche_global_2021} and in the case of less singular potentials in~\cite{lafleche_propagation_2019}, which was extended in~\cite{chong_global--time_2022} for the Hartree--Fock equation.

	\subsection{Previous results}\label{sec:previous_results}
	
	Our objective is to improve the rate of convergence in Wasserstein distances from the Hartree equation with Coulomb potential to the Vlasov--Poisson equation in the semiclassical limit, that is when the Planck constant becomes negligible. This semiclassical limit is also rather well understood. It was first proved in a weak sense by Lions and Paul in~\cite{lions_sur_1993} for general interaction potentials, including the Coulomb and gravitational ones. Quantitative rates of convergence were then obtained for sufficiently smooth potentials in strong topologies~\cite{athanassoulis_strong_2011-1, amour_classical_2013, amour_semiclassical_2013, benedikter_hartree_2016}. The semiclassical limit was also investigated in the case of positive densities using relative entropy methods in \cite{lewin_hartree_2020}.
	
	In~\cite{golse_mean_2016}, Golse, Mouhot and Paul introduced a quantum analogue of the $2$-Wasserstein distance to study the mean-field limit from the Schr\"odinger equation to the Hartree equation uniformly in the Planck constant, and then in~\cite{golse_schrodinger_2017}, Golse and Paul introduced a semiclassical analogue of the $2$-Wasserstein distance from which they obtained a quantitative estimate for potentials with Lipschitz forces. This was extended to singular potentials including the Coulomb and gravitational potentials in~\cite{lafleche_propagation_2019} and to the case of the presence of a magnetic field in~\cite{porat_magnetic_2022}. For singular potentials, quantitative rates of convergence were then also obtained in stronger topologies in~\cite{saffirio_semiclassical_2019, saffirio_hartree_2020, lafleche_strong_2023, chong_l2_2023}. A review of these recent results can be found in~\cite{lafleche_uniqueness_2023}.
	
	These latter results in strong topologies allow to have results that are global in time in the sense that they hold for any arbitrary large time, but they require more regular data for the limiting equation than the techniques using Wasserstein distances. On the other hand, for the Coulomb and gravitational potentials, the results in~\cite{lafleche_propagation_2019} are only local in time as they depend on the uniform-in-$\hbar$ propagation of moments for the Hartree equation. Moreover, on the interval $[0,T]$ on which the propagation of moments holds, as was computed more precisely in~\cite{lafleche_global_2021}, the rate of convergence in terms of $\hbar$ deteriorates very rapidly in time as it is of the form $\hbar^{\exp(-\lambda(t))}$ where $\lambda(t)$ is proportional to $1 + \sup_{t\in[0,T]}(\Nrm{\rho(t)}{L^\infty(\Rd)} + \Nrm{\rho_f(t)}{L^\infty(\Rd)})$.
	
\subsection{Notation}

	A typical object used to make the link between classical and quantum mechanics is the Weyl quantization~\cite{weyl_quantenmechanik_1927}, that associates to a function $f = f(x,\xi)$ of the phase space, an operator
	\begin{equation}\label{eq:Weyl_def_0}
		\op_f := \iintd \widehat{f}(y,\xi) \,e^{2i\pi\(y\cdot x + \xi\cdot \opp\)}\d y \d\xi
	\end{equation}
	by analogy with the Fourier inversion formula. Here, $(y,\xi)\in\Rdd$ is the phase space variable, $\opp = -i\hbar\nabla$ is the quantum analogue of the momentum, $x$ is the operator of multiplication by the variable $x$, and $\widehat{f}$ denotes the Fourier transform of $f$, defined by
	\begin{equation*}
		\F{f}\!(y,\xi) = \widehat{f}(y,\xi) = \iintd f(y',\xi')\, e^{-2i\pi\(y\cdot y'+\xi\cdot\xi'\)}\d y'\d \xi'. 
	\end{equation*}
	An elementary computation (see e.g.~\cite[Chapter~2]{folland_harmonic_1989}) shows that $\op_f$ is the operator with integral kernel
	\begin{equation*} 
		\op_f(x,y) = \intd e^{-2i\pi\(y-x\)\cdot\xi} \, f(\tfrac{x+y}{2},h\xi)\d\xi.
	\end{equation*}
	Another widely used object is the Wigner transform~\cite{wigner_quantum_1932}, which is the inverse operation: it associates to an operator $\op$ a function on the phase space
	\begin{equation*} 
		f_{\op}(x,\xi) := \intd e^{-i\,y\cdot\xi/\hbar} \,\op(x+\tfrac{y}{2},x-\tfrac{y}{2})\d y.
	\end{equation*}
	One can then verify that $\op_{f_{\op}} = \op$. 
	In general, these pointwise identities should be understood in the sense of distributions. The Wigner transform verifies $\intd f_{\op}\d \xi = \rho(x)$ where $\rho$ was defined by Equation~ \eqref{eq:spatial_density_quantum}, and so in particular 
	\begin{equation}\label{eq:Wigner_mass}
		\iint f_{\op}\d x\d\xi = \intd \rho(x)\d x = \hd\Tr{\op} = 1.
	\end{equation}
	The Wigner transform of a density operator is always a real-valued function, but it is not in general a nonnegative function. It is however also possible to associate to any density operator $\op$ a probability distribution on the phase space called the Husimi transform of $\op$, and defined by
	\begin{equation*}
		\tilde{f}_{\op} := g_h * f_{\op} \geq 0
	\end{equation*}
	with $g_h(z) = \(\pi\hbar\)^{-3} e^{-\n{z}^2/\hbar}$ for $z\in\Rdd$. Similarly, the Weyl quantization of a nonnegative function $f$ of the phase space is always self-adjoint but not always nonnegative, and one can define an associated Toeplitz operator, also called (anti)-Wick quantization (see e.g.~\cite{lerner_feffermanphong_2007}) or averaging of coherent states (see e.g.~\cite{lions_sur_1993}) by
	\begin{equation*}
		\widetilde{\op}_f := \intdd f(z) \, \op_{g_h(z-\cdot)} \d z = \op_{\tilde{f}}
	\end{equation*}
	where $\tilde{f} = g_h * f$. It can be seen as a quantum convolution by a Gaussian converging to $\delta_0$ (see e.g.~\cite{lafleche_quantum_2022}).
	
	The Wigner transform of density operators being in general not nonnegative, the first equality in Equation~\eqref{eq:Wigner_mass} should in general be interpreted in a weak sense since $f_{\op}$ might not be integrable. However, $f_{\op}$ is always square integrable and the following identity holds
	\begin{equation*}
		\Nrm{f_{\op}}{L^2(\Rd)}^2 = \hd \Tr{\op^2},
	\end{equation*}
	where the right-hand side, a scaled squared Hilbert--Schmidt norm, is always finite since $\op$ is trace-class. We refer the reader to \cite{folland_harmonic_1989, lions_sur_1993} for more properties on the Wigner and Husimi transforms.
	
	In general, we will define the semiclassical generalization of the Lebesgue norms on the phase space in terms of scaled Schatten norms for $p\in[1,\infty)$ by
	\begin{align}\label{eq:def_norm}
		\Nrm{\op}{\L^p} := h^{\frac{3}{p}} \Nrm{\op}{p} = \(\hd \Tr{\n{\op}^p}\)^\frac{1}{p}.
	\end{align}
	and by setting $\Nrm{\op}{\L^\infty} = \Nrm{\op}{\infty}$ to be the operator norm when $p=\infty$. We will also denote the space of density operators with $n$ bounded velocity moments as follows
	\begin{equation*}
		\opP_n := \set{\op\in\opP, \hd\Tr{\n{\opp}^n\op} < \infty}.
	\end{equation*}

\subsection{Quantum optimal transport pseudometrics}

	Our result follows the strategy introduced by Golse and Paul in~\cite{golse_schrodinger_2017} to use a semiclassical analogue of the Wasserstein distance. One first defines the set of couplings of a probability density $f$ and a trace-class operator $\op$ as the set of measurable maps $(x,\xi)\mapsto \opgam(x,\xi)$ from ${\Rdd}$ to the space of trace-class operators such that
	$$
	\Tr{ \opgam(x,\xi)}=f(x,\xi)\quad \text{for a.e. }(x,\xi) \in \Rdd,\qquad \int_{\Rdd}\opgam(x,\xi) \d x \d \xi =\op.
	$$
	The set of all such couplings is denoted by $\cC(f,\op)$, and the pseudometric is then defined as follows:
	\begin{equation}\label{eq:Wh}
		\Wh(f,\op)^2 := \inf_{\opgam\in\cC(f,\op)} \intdd \hd\Tr{\opc(z)\,\opgam(z)} \d z
	\end{equation}
	with $z=(y,\xi)$ and $\opc(z) = \n{y-x}^2 + \n{\xi-\opp}^2$. The term pseudometric is chosen because these quantities are not distances, as they can never be $0$ since they verify
	\begin{equation*}
		\Wh(f,\op)^2 \geq 3\,\hbar.
	\end{equation*}
	However, they still behave similarly as distances when $\hbar\to 0$. In particular, as proved in~\cite[Theorem~4.2]{golse_quantum_2021}, the following triangle inequality holds:
	\begin{align*}
		\Wh(f,\op) &\leq W_2(f,g) + \Wh(g,\op).
	\end{align*}
	Using the previously defined Husimi transforms, these pseudometrics can also be compared with the classical Monge--Kantorovich--Wasserstein distances. It was proved in~\cite[Theorem~2.4]{golse_schrodinger_2017} that
	\begin{equation}\label{eq:Wh_bound_below}
		\Wh(f,\op)^2 \geq W_2(f,\tilde{f}_{\op})^2 - 3\,\hbar
	\end{equation}
	while in the special case of Toeplitz operators
	\begin{equation*}
		\Wh(f,\tilde{\op}_g)^2 \leq W_2(f,g)^2 + 3\,\hbar.
	\end{equation*}
	In the case of more general operators, it was proved in~\cite{lafleche_quantum_2023} that
	\begin{equation}\label{eq:Wh_bound_above}
		\Wh(f,\op) \leq W_2(f,\tilde{f}_{\op}) + \sqrt{3\,\hbar + \hbar^2 \,\sNrm{\nabla f_{\sqrt{\op}}}{L^2(\Rdd)}^2}.
	\end{equation}
	This implies in particular that, at least for sufficiently smooth operators, $\Wh(f,\op)$ is of the same order as $W_2(f,\tilde{f}_{\op})$ up to a term of order $\sqrt{\hbar}$. Actually, the quantity $\sqrt{\hbar}\, \sNrm{\nabla f_{\sqrt{\op}}}{L^2(\Rdd)}$ is bounded uniformly in $\hbar\in(0,1)$ even for states converging to discontinuous functions of the phase space, such as spectral functions of the form $\op = \indic_{\R_-}(-\hbar^2\Delta + U(x))$ for some smooth and confining potential $U$ (see e.g.~\cite{fournais_optimal_2020, deleporte_universality_2021, lafleche_optimal_2023}). 
	
\section{Main results}

	We compare solutions of the Vlasov and Hartree equations with bounded densities, that is $\rho_f,\rho\in L^\infty_\loc(\R_+,L^\infty(\Rd))$. To guarantee $\rho_f\in L^\infty_\loc(\R_+,L^\infty(\Rd))$, we can rely on the following classical result from~\cite{pallard_moment_2012}, where $\P(\Rdd)$ denotes the space of probability measures on $\Rdd$.
	\begin{theorem}[\cite{pallard_moment_2012}, Theorem~2]\label{thm:pallard}
		Given $k>2,$ let $f^\init\in \P\cap L^\infty(\Rdd)$ be a non-negative initial datum satisfying the following two assumptions
		\begin{equation*}
			\intdd \n{\xi}^k f^\init(x,\xi) \d x \d\xi < \infty,
		\end{equation*}
		and for any $R>0$ and $T>0$ it holds
		\begin{equation*}
			\supess \{f^\init(y-t\,\xi,w): \n{y-x} \leq Rt^{3/2}, \n{w-\xi}\leq Rt^{1/2}\}\in L^\infty([0,T]\times \Rd_x; L^1(\Rd_\xi)).
		\end{equation*}
		Then, there exists a unique weak solution $f\in C^0(\R_+; L^\infty(\Rdd)-w*)$ to the Cauchy problem for the $3$-dimensional Vlasov--Poisson system such that, for any $T>0$, we have
		\begin{equation*}
			\sup_{t\in[0,T]} \Nrm{\rho_f(t, \cdot)}{L^\infty(\Rd)} < \infty.
		\end{equation*}
	\end{theorem}
	
	As was indicated in Section~\ref{sec:previous_results} in introduction, there is no global-in-time uniform-in-$\hbar$ known analogous estimate for the quantum position density $\rho$, but a possibly-$\hbar$-dependent estimate is available for all times for example in~\cite{illner_global_1994}, which can be stated in the following way.
	\begin{theorem}[\cite{illner_global_1994}, theorems~3.10 and~3.11]
		There exists a unique solution $\op\in C^0(\R_+,\opP\cap\opP_4)\cap C^1(\R_+,\opP)$ with initial data $\op^\init\in\opP\cap\opP_4$ to the Hartree equation~\eqref{eq:Hartree} with Coulomb or gravitational potential. It verifies for any $T>0$,
		\begin{equation*}
			\sup_{t\in[0,T]} \Nrm{\rho(t,\cdot)}{L^\infty(\Rd)} < \infty.
		\end{equation*}
	\end{theorem}

	On the other hand, a finite-time uniform-in-$\hbar$ estimate can be found in \cite{lafleche_propagation_2019}. It gives the following.
	\begin{theorem}[\cite{lafleche_propagation_2019}, Theorem~5]\label{thm:moments}
		Let $\op$ be the solution to the Hartree equation~\eqref{eq:Hartree} as given in the previous theorem with initial condition $\op^\init\in\opP$ such that
		\begin{equation}\label{eq:condition_op_init}
			\hd\Tr{\n{\opp}^{16}\op^\init} \quad\text{ and } \quad \Nrm{\(1+\n{\opp}^4\)\op^\init}{\L^\infty}
		\end{equation}
		are bounded uniformly in $\hbar > 0$. Then there exists $T>0$ and $C>0$ independent of $\hbar$ such that
		\begin{equation*}
			\sup_{t\in[0,T]} \Nrm{\rho(t,\cdot)}{L^\infty(\Rd)} \leq C.
		\end{equation*}
	\end{theorem}

	Hence, under the conditions of the three above theorems, the following quantity is bounded for any time $t\geq 0$ and uniformly in $\hbar$ up to some time $T>0$
	\begin{equation}\label{eq:Cinfty}
		C_\infty(t) = \max\!\(1,\kappa\) \max\!\(1,\sNrm{\rho_{f(t)}}{L^\infty},\Nrm{\rho(t)}{L^\infty}\).
	\end{equation}
	
	We are now ready to state our main theorem.

	\begin{theorem}\label{thm:main}
		Let $f$ be a solution of the Vlasov equation~\eqref{eq:Vlasov} with initial condition $f^\init\in\P(\Rdd)$ verifying the hypotheses in Theorem~\ref{thm:pallard} and let $\op$ be a solution of the Hartree equation with initial condition $\op^\init\in\opP$ verifying the hypotheses in Theorem~\ref{thm:moments}. Then, for any $\eps\in(0,1)$ and any $t\geq 0$,
		\begin{equation*}
			\Wh(f(t),\op(t)) < C_\infty^{1/6}  \(C_\infty^\init\)^{\frac{1}{3}} \max\!\(\tfrac{5}{\eps}\, \Wh(f^\init,\op^\init)^{1-\eps}, 3\, \Wh(f^\init,\op^\init)\) e^{\Lambda_\eps(t)}
		\end{equation*}
		where $C_\infty$ is given by Equation~\eqref{eq:Cinfty}, $C_\infty^\init = C_\infty(0)$ and
		\begin{align*}
			\Lambda_\eps(t) &= 2\, \Lambda(t) + \frac{\(1-\eps\)\Lambda(t)^2}{\eps} &\text{ with } && \Lambda(t) &= \int_0^t \sqrt{C_\infty(s)} \d s.
		\end{align*}
	\end{theorem}
	
	\begin{remark}\label{rmk:h bar}
		The previous theorem gives an almost-optimality of the rate of convergence in $\hbar$ for the semiclassical limit in the following sense. Assuming that initially $\Wh(f^\init,\op^\init) \leq C\,\sqrt{\hbar} \leq 1$ and fixed $\eps>0$, it yields for any time $t\in[0,T]$ such that $C_\infty\in L^\infty[0,T]$ uniformly-in-$\hbar$, 
		\begin{equation*}
			\sqrt{3\,\hbar} \leq \Wh(f(t),\op(t)) \leq C_{T,\eps} \,\sqrt{\hbar}^{1-\eps}
		\end{equation*}
		for some constant $C_{T,\eps}$ independent of $\hbar$. More precisely, it follows from the proof that for $T>0$ sufficiently small, there exists a constant $C>0$ independent of $\hbar$ such that
		\begin{equation*}
			\sqrt{3\,\hbar} \leq \Wh(f(t),\op(t)) \leq C_T \,\sqrt{\hbar} \, e^{C\sqrt{\n{\ln \hbar}}}.
		\end{equation*}
	\end{remark}
	
	\begin{remark}
		Thank to the above inequalities~\eqref{eq:Wh_bound_below} and~\eqref{eq:Wh_bound_above}, the previous theorem yields the following convergence of the Husimi transform in the classical Wasserstein distances. Assuming $W_2(f,\tilde{f}_{\op}) \leq 1$ and $\sqrt{3\,\hbar}\leq 1$, then
		\begin{equation*}
			W_2(f,\tilde{f}_{\op}) <  \tfrac{C\,e^{\Lambda_\eps(t)}}{\eps}\( W_2(f^\init,\tilde{f}_{\op^\init})^{1-\eps} + \hbar^\frac{1-\eps}{2} + \hbar^{1-\eps} \,D_{\op^\init}\)
		\end{equation*}
		where $C = 5\,\sqrt{3}\,C_\infty^{1/6}  \(C_\infty^\init\)^{\frac{1}{3}}$ and $D_{\op} = \sNrm{\nabla f_{\sqrt{\op}}}{L^2(\Rdd)}$. In particular, if $f^\init$ is such that $W_2(f^\init,\tilde{f}_{\op^\init}) \leq \hbar^\frac{1-\eps}{2}$, then, as proved in \cite{lafleche_quantum_2023}, the convergence is of order $\hbar^\frac{1-\eps}{2}$ for typical states such as ground states of systems of Fermions without interaction at zero or positive temperature.
	\end{remark}

\section{Improved Loeper estimates}

	\subsection{Heuristics in the classical case}
	
	To motivate the proof of our main result, we first revisit briefly the argument introduced in~\cite{iacobelli_new_2022} and that we will adapt below to the quantum case. 
	
	\begin{prop}
		Let $f_1$ and $f_2$ be two solutions of the Vlasov--Poisson equation~\eqref{eq:Vlasov} with initial data $f_1^\init, f_2^\init\in\P(\Rdd)$ verifying the hypotheses in Theorem~\ref{thm:pallard} and such that $W_2(f_1,f_2) \leq \eta \in (0,1)$. Then there exists $T>0$ and $C>0$ independent of $\eta$ such that for any $t\in[0,T]$,
		\begin{equation}\label{eq:classical_stability}
			W_2(f_1,f_2) \leq \sqrt{\eta} \,e^{C\sqrt{\n{\ln \eta}}}.
		\end{equation}
	\end{prop}
	
	\begin{remark}
		Notice that Equation~\eqref{eq:classical_stability} can also be written $W_2(f_1,f_2) \leq \eta^\frac{1-\theta(\eta)}{2}$ where $\theta(\eta) = C_T/\sqrt{\n{\ln \eta}}$ is small when $\eta$ is small.
	\end{remark}
	
	\begin{proof}[Sketch of the proof]
		For some coupling $\gamma_0$ of two solutions $f_1$ and $f_2$ of the Vlasov--Poisson equation and for $X_k(t)$ and $P_k(t)$ solutions of the Newton's equations $\dot{X}_k(t) = P_k(t)$ and $\dot{P}_k(t) = E_{f_k}(X_k)(t)$ with initial conditions $(X_k(0),P_k(0)) = (x_0,p_0)$, we define $Q_{X^2} = \int_{\R^{12}} \n{X}^2\gamma_0(\d z_1\d z_2)$, $Q_{X\cdot P} = \int_{\R^{12}} X\cdot P\,\gamma_0(\d z_1\d z_2)$ and $Q_{P^2} = \int_{\R^{4d}} \n{P}^2\gamma_0(\d z_1\d z_2)$ with $X = X_1-X_2$, $P = P_1-P_2$ and $E = E_{f_1}(X_1)-E_{f_2}(X_2)$. Then defining $Q_\lambda := \lambda\,Q_{X^2} + Q_{P^2}$ for some time dependent function $\lambda$, we get 
		\begin{align*}
			\dpt Q_\lambda =  \lambda'\,Q_{X^2} + 2\,\lambda\,Q_{X\cdot P} + 2\,Q_{P\cdot E}.
		\end{align*}
		where $\lambda'= \dt \lambda$. 
		Now, as shown in \cite[Section 3.3]{iacobelli_new_2022}, using the Loeper-type estimate 
		\begin{equation*}
			Q_{E^2} \leq C_T\,Q_{X^2} \weight{\ln Q_{X^2}}^2,\quad \text{ where } \weight{y}=\sqrt{1+y^2}
		\end{equation*}
		(compare also with the proof of Lemma~\ref{lem:log-lip_2} below), we can bound
		\begin{align*}
			\dpt Q_\lambda \leq  \lambda'\,Q_{X^2} + 2 \,Q_{P^2}^{1/2}\,Q_{X^2}^{1/2}\(\lambda + C_T \weight{\ln Q_{X^2}}\).
		\end{align*}
		Hence since $Q_{P^2} \leq Q_\lambda$ and $Q_{X^2} \leq \frac{Q_\lambda}{\lambda}$, we obtain
		\begin{align*}
			\dpt Q_\lambda &\leq  \frac{\lambda'}{\lambda}\,Q_{\lambda} + 2 \,Q_{\lambda}\(\sqrt \lambda + \frac{C_T}{\sqrt \lambda} \weight{\ln(Q_\lambda/\lambda)}\).
		\end{align*}
		Now, slightly differently from \cite{iacobelli_new_2022}, rather than considering $Q_\lambda$ we consider  $Q := Q_\lambda/\lambda$. Then the bound above gives
		\begin{equation*}
			\dpt Q \leq 2 \,Q\(\sqrt \lambda + \frac{C_T}{\sqrt \lambda} \weight{\ln(Q)}\).
		\end{equation*}
		The idea now is to optimize with respect to $\lambda$ the term in parenthesis. This corresponds to choosing $\lambda$ as a function of $Q$, namely
		$\lambda = C_T \weight{\ln Q}$. As shown in \cite[Lemma 3.7]{iacobelli_new_2022} (see also the generalizations in  \cite[Section 4]{iacobelli_new_2022}), this choice of $\lambda(t)$ is indeed admissible as long as $Q(t)\leq 1$, and it leads to the bound
		\begin{align*}
			\dpt Q &\leq 4 \, Q\, \sqrt{C_T \weight{\ln Q}}.
		\end{align*}
		Now, define $\eta := Q(0)$. Then assuming that $\eta<1$ and solving the ODE $\dpt Y = 2\,C\,Y\,\sqrt{\n{\ln Y}}$ with $Y(0)=Q(0)$ for $t\in[0,T]$ where $T$ is so that $Y(t) \leq 1$ we get 
		\begin{equation*}
			\dpt \n{\ln Y} = -2\,C\,\sqrt{\n{\ln Y}}.
		\end{equation*}
		Therefore $\n{\ln Y} = \(\sqrt{\n{\ln\eta}} - C\,t\)^2$ and so for any $t\in[0,T]$
		\begin{equation*}
			Q(t) \leq Y(t) = e^{-\(\sqrt{\n{\ln\eta}} - C\,t\)^2}  = e^{\ln\eta - C^2\,t^2 + 2\sqrt{\n{\ln\eta}}\,C\,t} = \eta\, e^{2\,\sqrt{\n{\ln\eta}}\,C\,t}\, e^{-C^2\,t^2}.
		\end{equation*}
		In particular when $\eta$ is small, since on $\R_+$, $x\mapsto \weight{\ln(x)} x$ is an increasing function and $x^2 \,e^x \leq e^{Cx}$, it follows that
		\begin{equation*}
			\weight{\ln Q(t)} Q(t) \leq \weight{\ln Y(t)} Y(t) \leq \eta\, e^{C_T\sqrt{\n{\ln\eta}}}.
		\end{equation*}
		Since $\lambda\geq 1$, $W_2(f_1,f_2)^2\leq Q_{X^2+P^2} \leq \lambda\, Q$, which finishes the proof.
	\end{proof}
	
\subsection{Quantum case}

	Let us start with the following Lemma which indicates that the semiclassical pseudometric~\eqref{eq:Wh} controls the classical $2$-Wasserstein distance of the position densities.

	\begin{lem}\label{lem:position_vs_kinetic_W2}
		Let $f\in\P$, $\op \in \opP$ and define the associated position densities $\rho_f = \intd f\d\xi$ and $\rho(x) = \hd \op(x,x)$. Then
		\begin{equation}
			W_2(\rho_f,\rho) \leq \Wh(f,\op).
		\end{equation}
	\end{lem}
	
	
	\begin{proof}
		Let $\opgam\in\cC(f,\op)$ and define $\!\rho_{\opgam}(x_0,x) := \hd \intd \opgam(z_0)(x,x)\d\xi_0$ in the weak sense as in Equation~\eqref{eq:spatial_density_quantum}  where $z_0 = (x_0,\xi_0)$ and $\opgam(z_0)(x,y)$ denotes the kernel of the operator $\opgam(z_0)$. Then $\rho_{\opgam}$ is a coupling of $\rho_f$ and $\rho$ since
		\begin{align*}
			\intd \rho_{\opgam}(x_0,x)\d x_0 &= \hd \intdd \opgam(z_0)(x,x)\d z_0 = \hd\op(x,x) = \rho(x)
			\\
			\intd \rho_{\opgam}(x_0,x)\d x &= \hd \intd \Tr{\opgam(z_0)}\d \xi_0 = \intd f(z_0)\d\xi_0 = \rho_f(x_0).
		\end{align*}
		Therefore,
		\begin{align*}
			W_2(\rho_f,\rho)^2 &\leq \intdd \n{x_0-x}^2 \rho_{\opgam}(\d x_0\d x) = \hd \intdd \Tr{\n{x_0-x}^2 \opgam(z_0)}\d z_0
			\\
			&\leq \hd \intdd \Tr{\opc(z_0)\, \opgam(z_0)}\d z_0,
		\end{align*}
		and so the result follows by taking the infimum with respect to all the semiclassical couplings $\opgam \in\cC(f,\op)$.
	\end{proof}
	
	The next lemma about the log-Lipschitz continuity of the Coulomb potential is also classical (see e.g. \cite{loeper_uniqueness_2006}). We give here a quantitative version.
	\begin{lem}\label{lem:log-lip}
		Let $K$ be given by Equation~\eqref{eq:Coulomb}. Then
		\begin{equation*}
			\n{\nabla K*\rho(x)-\nabla K*\rho(y)} \leq 2\,\cC_\infty\,\n{x-y} \(c + \ln_+\!\(\frac{R^2}{\n{x-y}^2}\)\)
		\end{equation*}
		where $\cC_\infty \geq \kappa \Nrm{\rho}{L^\infty}$, $R^3 = \frac{3\,\kappa\,\Nrm{\rho}{L^1}}{8\,\pi\,\theta_1^3\, \cC_\infty}$ with $\theta_1 = \frac{3-\sqrt{5}}{2}$ and $c = \frac{2}{3}+\frac{\sqrt{5}-\theta_2}{2}$ with $\theta_2$ the solution of $\theta_2^2 = (1-\theta_2)^3$. The constant $c$ verifies $c< 1.57$.
	\end{lem}
	From this estimate, we obtain the following inequality.
	\begin{lem}\label{lem:log-lip_2}
		With the notation of Lemma~\ref{lem:log-lip}, let $\Phi$ be the positive, concave and increasing function on $\R_+$ defined by
		\begin{equation*}
			\Phi(x) = \left\{\begin{aligned}
				&x \(c - \ln(x)\)^2 &\text{ if } x \leq x_0
				\\
				&\Phi(x_0) + c^2 \(x-x_0\) &\text{ if } x > x_0
			\end{aligned}\right.
		\end{equation*}
		where $\ln(x_0) = c-1-\sqrt{1+c^2}$ (it implies $x_0\in(0.27,0.28)$). Then for any $\gamma\in\P(\Rdd)$,
		\begin{equation*}
			\intdd \n{\nabla K * \rho(x) - \nabla K * \rho(y)}^2 \gamma(\d x\d y) \leq \(2\,\cC_\infty\)^2 R^2\, \Phi(Q_X/R^2)
		\end{equation*}
		where
		\begin{equation*}
			Q_X := \intdd \n{x-y}^2 \gamma(\d x\d y).
		\end{equation*}
	\end{lem}

	\begin{lem}\label{lem:diff_ineq}
		Assume $\rho = h^3 \op(x,x) \in L^1\cap L^\infty$ and $\rho_f = \intd f\d \xi \in L^\infty$. Then, with the notations of Lemma~\ref{lem:log-lip_2}, for any smooth $\lambda = \lambda(t) > 0$,
		\begin{equation*}
			Q(t) := \frac{1}{\lambda R^2}\intdd \hd\Tr{\(\lambda(t)\n{x-y}^2+\n{\xi-\opp}^2\)\opgam(z)} \d z
		\end{equation*}
		verifies
		\begin{equation*}
			\dpt Q \leq Q \(\sqrt{\lambda} + C_\infty \frac{1 + 2\sqrt{\Phi(Q)/Q}}{\sqrt{\lambda}}\),
		\end{equation*}
		where $C_\infty(t) \geq \kappa \max(\sNrm{\rho_{f(t)}}{L^\infty},\Nrm{\rho(t)}{L^\infty})$ and $R^3 = \frac{3\,\kappa\,\Nrm{\rho}{L^1}}{8\,\pi\,\theta_1^3\, C_\infty}$.
	\end{lem}
	
	\begin{proof}
		Let $\opgam(z) = \opgam(t,z)$ be such that $\opgam(0,\cdot) = \opgam^\init$ is the optimal coupling between $f^\init$ and $\op^\init$ and
		\begin{equation*}
			\dpt \opgam = -(\xi\cdot\Dx+E_f\cdot\Dv)\opgam + \frac{1}{i\hbar} \com{H,\opgam}.
		\end{equation*}
		Then $\opgam\in\cC(f(t),\op(t))$ (see e.g. \cite[Lemma~4.2]{golse_schrodinger_2017}). We define $Q_\lambda := \lambda(t)\, Q_X + Q_P$ with
		\begin{align*}
			Q_{X} := \intdd \hd\Tr{\n{x-y}^2\opgam(z)} \d z
			\\
			Q_{P} := \intdd \hd\Tr{\n{\xi-\opp}^2\opgam(z)} \d z
		\end{align*}
		with $z=(x,\xi)$ and $\opp = -i\hbar\nabla_y$. Then by a direct computation detailed in \cite[Section~4.3]{golse_schrodinger_2017}, one can write the time derivative of the above quantities in the following way
		\begin{equation*}
			\dpt Q_{X} = \intdd \hd\Tr{\(\(x-y\)\cdot\(\xi-\opp\)+\(\xi-\opp\)\cdot\(x-y\)\)\opgam(z)} \d z
		\end{equation*}
		and 
		\begin{align*}
			\dpt Q_{P} &= \intdd \hd\Tr{\(\(E_f(x)-E_{\op}(y)\)\cdot\(\xi-\opp\)+\(\xi-\opp\)\cdot\(E_f(x)-E_{\op}(y)\)\)\opgam(z)} \d z.
		\end{align*}
		By a variant of the Cauchy--Schwartz inequality for the trace of the form $\Tr{A^*B\,\opgam}^2 \leq \tr(\n{A}^2\opgam) \tr(\n{B}^2\opgam)$, it yields
		\begin{equation*}
			\dpt Q_\lambda \leq \lambda' \,Q_X + 2\,\lambda\, Q_{X}^{1/2}\,Q_{P}^{1/2} + 2\, Q_{P}^{1/2}\,Q_{E}^{1/2}
		\end{equation*}
		where $\lambda'$ denotes the derivative of $\lambda$ with respect to time and
		\begin{equation*}
			Q_E = \intdd \hd\Tr{\n{E_f(x) - E_{\op}(y)}^2\opgam(z)} \d z.
		\end{equation*}
		Denoting by $\opgam(z)(y_1,y_2)$ the integral kernel of $\opgam(z)$, we get
		\begin{equation*}
			Q_E = \hd \intdd \intd \n{E_f(x)-E_{\op}(y)}^2 \opgam(z)(y,y)\d y \d z
		\end{equation*}
		and since $\opgam(z)(y,y)\geq 0$ as the diagonal of the kernel of a positive operator, by the triangle inequality
		\begin{equation*}
			Q_E^{1/2} \leq I_1 + I_2
		\end{equation*}
		where 
		\begin{align*}
			I_1^2 &= \intd \n{E_f(x)-E_{\op}(x)}^2 \rho_f(x)\d x
			\\
			I_2^2 &= \hd \intdd \intd \n{E_{\op}(x)-E_{\op}(y)}^2 \opgam(z)(y,y)\d y \d z
		\end{align*}
		using the fact that $\hd \intd \opgam(z)(y,y)\d y = f(z)$. By Loeper's inequality \cite[Theorem~2.7]{loeper_uniqueness_2006},
		\begin{equation*}
			I_1 \leq \kappa\,\max(\Nrm{\rho_f}{L^\infty},\Nrm{\rho}{L^\infty})^\frac{1}{2} \Nrm{\rho_f}{L^\infty}^\frac{1}{2} W_2(\rho_f,\rho).
		\end{equation*}
		Therefore, by Lemma~\ref{lem:position_vs_kinetic_W2}, and since $\opgam$ is a coupling of $f$ and $\op$, it gives
		\begin{equation}\label{eq:bound_I1}
			I_1 \leq C_\infty \,Q_X^{1/2}.
		\end{equation}
		On the other side, by Lemma~\ref{lem:log-lip_2},
		\begin{equation}\label{eq:bound_I2}
			I_2 \leq 2\,C_\infty\, R\,\sqrt{\Phi(Q_X/R^2)}.
		\end{equation}
		Therefore, combining Inequality~\eqref{eq:bound_I1} and Inequality~\eqref{eq:bound_I2}, we get
		\begin{equation*}
			Q_E^{1/2} \leq C_\infty\(\sqrt{Q_X} + R\,\sqrt{\Phi(Q_X/R^2)}\).
		\end{equation*}
		That is, coming back to the inequality for $\dpt Q_\lambda$,
		\begin{align*}
			\dpt Q_\lambda &\leq \lambda' \,Q_X + 2\,Q_P^{1/2}\(\lambda\, Q_X^{1/2} + C_\infty \,Q_X^{1/2} + 2\,C_\infty\,R\, \sqrt{\Phi(Q_X/R^2)}\)
			\\
			&\leq \frac{\lambda'}{\lambda} \,Q_\lambda + Q_{\lambda} \(\sqrt{\lambda} + \frac{C_\infty}{\sqrt{\lambda}} + 2\,C_\infty\, \sqrt{\tfrac{R^2}{Q_\lambda}\Phi(\tfrac{Q_\lambda}{\lambda\,R^2})}\)
		\end{align*}
		where we used the fact that $Q_X \leq \lambda^{-1} Q_\lambda$ and $(Q_PQ_X)^{1/2} \leq Q_\lambda/(2\sqrt{\lambda})$.
	\end{proof}
	
	It now suffices to study the ordinary differential equation associated to the above lemma. This gives the following inequalities.
	\begin{lem}\label{lem:gronwall_ineq}
		With the notations of Lemma~\ref{lem:diff_ineq}, there exists $\lambda = \lambda(t) > 0$ such that
		\begin{equation*}
			Q(t) \leq Q(0)\,e^{2\Lambda_1(t)\sqrt{(c_2-\ln Q(0))_+} + 2\Lambda_2(t)}
		\end{equation*}
		where $\Lambda_1(t) = \Lambda(\min(t,\tau))$ and
		\begin{equation*}
			\Lambda_2(t) = \sqrt{c_2} \(\Lambda(t)-\Lambda(\tau)\)_+ - \Lambda_1(t)^2/2
		\end{equation*}
		with
		\begin{align*}
			\Lambda(t) &= \int_0^t \sqrt{2\,C_\infty(s)} \d s
			\\
			\tau &= \Lambda^{-1}\!\(\sqrt{(c_2-\ln Q(0))_+} - \sqrt{c_2+y_0}\)_+
		\end{align*}
		where $y_0 = c-1-\sqrt{1+c^2} < 1.3$, $c_2 = c + \tfrac{1}{2} < 2.07$ and $\lambda$ verifies
		\begin{equation*}
			\lambda = C_\infty \(1+2 \sqrt{\Phi(Q)/Q}\),
		\end{equation*}
		with $c$ and $\Phi$ defined in Lemma~\ref{lem:log-lip_2}
	\end{lem}
	
	\begin{proof}
		Let
		\begin{equation*}
			\ell(q) := C_\infty \(1+2 \sqrt{\Phi(q)/q}\).
		\end{equation*}
		This is a convex decreasing function which converges to the constant $C_\infty\(1+2\,c\)$ when $q\to\infty$ and is such that $\ell(q) \sim 2\,C_\infty\,\n{\ln q}$ when $q\to 0$. Then by Lemma~\ref{lem:diff_ineq},
		\begin{equation*}
			\dpt Q \leq Q \(\sqrt{\lambda} + \frac{\ell(Q)}{\sqrt{\lambda}}\).
		\end{equation*}
		Optimizing with respect to $\lambda$, we take $\lambda := \ell(Q)$. More precisely, since for any $\alpha,\beta > 0$, there is always a solution to the equation $q = \alpha + \beta/\ell(q)$, we define $Q$ as the solution of $Q = Q_X + Q_P/\ell(Q)$. This gives $\dpt Q \leq 2\,Q \,\sqrt{\ell(Q)}$ and so
		\begin{equation*}
			\dpt \ln(Q) \leq 2 \sqrt{C_\infty} \sqrt{1 + 2\,\Psi(-\ln(Q))}
		\end{equation*}
		where $\Psi(y) = \sqrt{\Phi(e^{-y})\,e^{y}}$. Let $Y(t) := -\ln(Q(t))$, then this implies that
		\begin{equation*}
			Y' \geq -2 \sqrt{C_\infty} \sqrt{1 + 2\,\Psi(Y)}.
		\end{equation*}
		Notice that $\Psi(y) = c+y$ when $y\geq y_0$, so that
		\begin{align*}
			Y' \geq -2\,\sqrt{2\,C_\infty}\,\sqrt{c_2 + Y}
		\end{align*}
		where $c_2 = c+\frac{1}{2}$. Hence, assuming $Y(0)\geq y_0$, Gr\"onwall's lemma implies that
		\begin{equation}\label{eq:ineq_Y}
			Y(t)+c_2 \geq \(\sqrt{c_2+Y(0)}-\int_0^t \sqrt{2\,C_\infty(s)}\d s\)^2
		\end{equation}
		for any $t\in[0,t_0]$ where $t_0\in [0,\infty]$ is the first time such that $Y(t_0) = y_0$. In particular, this also holds for $t\leq \tau$ where $\tau \in [0,t_0]$ is the first time such that the right-hand side of Inequality~\eqref{eq:ineq_Y} takes the value $y_0+c_2$. If we already had $Y(0) \leq y_0$, then we set $\tau = 0$. That is, setting $\Lambda := \int_0^t \sqrt{2\,C_\infty(s)}\d s$, in general, we define
		\begin{equation*}
			\tau = \Lambda^{-1}\!\(\sqrt{\(c_2+Y(0)\)_+} - \sqrt{y_0+c_2}\)_+.
		\end{equation*}
		When $t \geq \tau$, then $Y(t)$ is larger than the solution $Z(t)$ to
		\begin{equation*}
			Z' = -2 \,\sqrt{C_\infty} \sqrt{1 + 2\,\Psi(Y)}
		\end{equation*}
		with initial condition $Z(\tau) = \min(y_0,Y(0))$. Since $Z$ is decreasing, $Z(t) \leq y_0$ for any $t \geq \tau$ and so $\Psi(y) \geq c$. The Gr\"onwall Lemma then implies $Y(t) \geq -2\,\sqrt{c_2} \(\Lambda(t)-\Lambda(\tau)\)$. For any $t\geq 0$, we thus obtained the following inequalities
		\begin{equation}\label{eq:Q_ineq_0}
			Q(t) \leq \left\{\begin{aligned}
				& e^{-(\sqrt{(c_2-\ln Q(0))_+} - \Lambda(t))^2+c_2} &&\text{ for } t < \tau
				\\
				& \max(x_0,Q(0)) \, e^{2\,\sqrt{c_2} \(\Lambda(t)-\Lambda(\tau)\)} &&\text{ for } t\geq \tau.
			\end{aligned}\right.
		\end{equation}
		When $t<\tau$, the above inequality can be written $Q(t) \leq Q(0)\,e^{2\Lambda(t)\sqrt{(c_2-\ln Q(0))_+} - \Lambda(t)^2}$. Notice that it follows from the definitions that $x_0 = e^{-(\sqrt{(c_2-\ln(Q(0)))_+} - \Lambda(\tau))^2+c_2}$, so that when $Q(0) > x_0$ and $t\geq \tau$, one gets 
		\begin{equation*}
			Q(t) \leq Q(0)\,e^{2\Lambda(\tau)\sqrt{(c_2-\ln(Q(0)))_+} - \Lambda(\tau)^2} e^{2\,\sqrt{c_2} \(\Lambda(t)-\Lambda(\tau)\)}
		\end{equation*}
		leading to the result.
	\end{proof}
	
	\begin{proof}[Proof of Theorem~\ref{thm:main}]
		Take $C_\infty(t) = \max(1,\kappa)\max(1,\Nrm{\rho(t)}{L^\infty},\sNrm{\rho_{f(t)}}{L^\infty})$. Then by the definitions of $Q$ and $\lambda$ given in the proof of Lemma~\ref{lem:gronwall_ineq}, one obtains that $\lambda \geq C_\infty\(1+2\,c\) \geq 1$ and
		\begin{equation*}
			Q_X + Q_P \leq R^2\,\lambda(Q)\,Q  \quad \text{ and } \quad R^2\,Q \leq Q_X + Q_P.
		\end{equation*}
		Moreover, by definition, $\Wh(f,\op)^2 \leq Q_X(t) + Q_P(t)$. Therefore, defining $W(t) := \Wh(f(t),\op(t))/R(t)$ and $w = W(0)/R(0)$, by Lemma~\ref{lem:gronwall_ineq} and then by minimizing over all couplings $\opgam\in\cC(f,\op)$, it holds
		\begin{equation*}
			W(t)^2 \leq \lambda\!\(w\,e^{2\Lambda_1(t)\sqrt{(c_2-\ln w)_+} + 2\Lambda_2(t)}\)  w\,e^{2\Lambda_1(t)\sqrt{(c_2-\ln w)_+} + 2\Lambda_2(t)}
		\end{equation*}
		and so if $t\leq \tau$
		\begin{equation*}
			W(t) \leq \sqrt{C_\infty}\(\sqrt{c_2-\ln w} - \Lambda(t)\) w\,e^{\Lambda_1(t)\sqrt{c_2-\ln w} + \Lambda_2(t)}.
		\end{equation*}
		Using Young's inequality for the product $a\,\sqrt{b} \leq \eps\,b + \frac{a^2}{4\,\eps}$ and the fact that for $\theta\leq 1$, $-\ln(\theta) \leq \theta^{-\epsilon}/(\epsilon\,e)$, this can be written
		\begin{equation*}
			W(t) \leq \sqrt{C_\infty}\, \frac{c_2^{\epsilon+\eps}}{\epsilon\,e} \, w^{1-\eps-\epsilon} \,e^{\Lambda_2(t) + \frac{\Lambda_1(t)^2}{4\eps}}.
		\end{equation*}
		This implies that for any $\eps>0$,
		\begin{equation*}
			W(t) \leq \sqrt{C_\infty}\, \frac{2\,c_2^{\eps}}{\eps\,e} \, w^{1-\eps} \,e^{\Lambda_2(t) + \frac{\Lambda_1(t)^2}{2\eps}}.
		\end{equation*}
		Recalling that $R(t)^3 = \frac{3\kappa\Nrm{\rho}{L^1}}{8\pi\theta_1^3 C_\infty}$ and assuming $\Nrm{\rho}{L^1} = 1$, it yields
		\begin{align*}
			\Wh(f(t),\op(t)) &\leq \frac{2}{\eps\,e}\, C_\infty^{1/6}  \(C_\infty^\init\)^{\frac{2\(1-\eps\)}{3}} \(\frac{3\,\kappa\,c_2^3}{8\,\pi\,\theta_1^3}\)^\frac{2\eps}{3} \Wh(f^\init,\op^\init)^{1-\eps} \,e^{\Lambda_2(t) + \frac{\Lambda_1(t)^2}{2\eps}}
			\\
			&< \frac{2}{\eps} \,C_\infty^{1/6} \(C_\infty^\init\)^{\frac{2}{3}} \Wh(f^\init,\op^\init)^{1-\eps} \,e^{\Lambda_2(t) + \frac{\Lambda_1(t)^2}{2\eps}}.
		\end{align*}
		If $t\geq \tau$, $\lambda$ is bounded by a constant and more precisely
		\begin{equation*}
			W(t) \leq \sqrt{2\,C_\infty \(c_2+\,y_0\)} \, w\,e^{\Lambda_1(t)\sqrt{(c_2-\ln w)_+} + \Lambda_2(t)}
		\end{equation*}
		so using again Young's inequality for the product
		\begin{align*}
			W(t) &\leq \sqrt{C_\infty}\,\sqrt{2\(c_2+y_0\)} \, \max(c_2^{\eps/2}\,w^{1-\eps/2},w)\, e^{\Lambda_2(t) + \frac{\Lambda_1(t)^2}{2\eps}}
			\\
			&\leq \sqrt{C_\infty}\,\sqrt{2\(c_2+y_0\)} \, \max(c_2^{\eps}\,w^{1-\eps},w)\, e^{\Lambda_2(t) + \frac{\Lambda_1(t)^2}{2\eps}}
		\end{align*}
		Hence, by the definition of $R$
		\begin{multline*}
			\Wh(f(t),\op(t)) \leq \sqrt{2\(c_2+y_0\)} \,C_\infty^{1/6}  \(C_\infty^\init\)^{\frac{1}{3}} \\ \max\!\(\!\(\tfrac{3\,c_2^3}{8\,\pi\,\theta_1^3}\)^\frac{\eps}{3}  \Wh(f^\init,\op^\init)^{1-\eps},\Wh(f^\init,\op^\init)\) e^{\Lambda_2(t) + \frac{\Lambda_1(t)^2}{2\eps}},
		\end{multline*}
		which yields the result using the fact that $\eps\in(0,1)$.
	\end{proof}

\appendix
\section{Explicit estimates for the Coulomb potential}

	In this section, we give a proof of Lemma~\ref{lem:log-lip}, allowing us both to be more self-contained and to give explicit upper bound to the constants. We start by the following elementary inequality.
	\begin{lem}
		For any $(x,y)\in(\Rd\setminus\set{0})^2$,
		\begin{equation}
			\\\label{eq:Coulomb_3}
			\n{\frac{x}{\n{x}^3} - \frac{y}{\n{y}^3}} \leq \frac{\n{y-x}}{\n{x}\n{y}}\(\frac{1}{\n{x}}+\frac{1}{\n{y}}\).
		\end{equation}
	\end{lem}

	\begin{proof} 
		Since
		\begin{align*}
			\n{\frac{x}{\n{x}^3} - \frac{y}{\n{y}^3}}^2 &= \frac{\n{y}^6\n{x}^2+\n{x}^6\n{y}^2 - 2\,\n{y}^3\n{x}^3\,x\cdot y}{\n{x}^6\n{y}^6}
			\\
			&= \frac{\n{y}^{4}+\n{x}^{4} - 2\,\n{y}\n{x}\,x\cdot y}{\n{x}^{4}\n{y}^{4}} = \frac{\n{\n{y}y - \n{x}x}^2}{\n{x}^{4}\n{y}^{4}}
		\end{align*}
		we deduce that
		\begin{equation*}
			\n{\frac{x}{\n{x}^3} - \frac{y}{\n{y}^3}} = \frac{\n{\n{y}y - \n{x}x}}{\n{x}^{2}\n{y}^{2}} = \frac{\n{\n{y}(y-x) + \(\n{y}-\n{x}\)x}}{\n{x}^2\n{y}^2}
		\end{equation*}
		which gives Inequality~\eqref{eq:Coulomb_3} by the triangle inequality.
	\end{proof}
	
	\begin{lem}\label{lem:log-lip_0}
		Let $\nabla K = \frac{x}{\n{x}^3}$. Then
		\begin{equation*}
			\n{\nabla K*\rho(x)-\nabla K*\rho(y)} \leq 8\pi\Nrm{\rho}{L^\infty} \n{x-y} \(c + \ln_+\!\(\frac{R^2}{\n{x-y}^2}\)\),
		\end{equation*}
		with $R^3 = \frac{3\,\Nrm{\rho}{L^1}}{8\pi\,\theta_1^3 \Nrm{\rho}{L^\infty}}$, $\theta_1 = \frac{3-\sqrt{5}}{2}$ and $c < 1.57$.
	\end{lem}
	
	\begin{proof}
		Using Inequality~\eqref{eq:Coulomb_3} yields
		\begin{multline*}
			\n{\nabla K * \rho(x) - \nabla K * \rho(y)} \leq \intd \rho(w) \n{\frac{x-w}{\n{x-w}^3} - \frac{y-w}{\n{y-w}^3}} \d w
			\\
			\leq \n{y-x} \(\intd \frac{\rho(w)}{\n{x-w}^2\n{y-w}} + \frac{\rho(w)}{\n{x-w}\n{y-w}^2} \d w\).
		\end{multline*}
		Let $I_{\Omega}(k) = \int_{\Omega} \frac{\rho(w)}{\n{x-w}^{3-k}\n{y-w}^{k}} \d w$, $\theta\in(0,1)$ and $\delta := \n{x-y}$. Then since $\n{x-w}\leq \theta\, \delta$ implies $\n{y-w} \geq \(1-\theta\) \delta$, we get
		\begin{equation*}
			I_{\n{x-w}\leq \theta \delta}(k) \leq \frac{4\pi}{\(1-\theta\)^k \delta^k} \Nrm{\rho}{L^\infty} \int_0^{\theta \delta} r^{-1+k} \d r \leq \frac{4\pi\,\theta^k}{k\(1-\theta\)^k} \Nrm{\rho}{L^\infty}.
		\end{equation*}
		On another side
		\begin{equation*}
			I_{\n{x-w},\n{y-w}\geq R}(k) \leq \frac{\Nrm{\rho}{L^1}}{R^3}.
		\end{equation*}
		Finally, for $\delta\theta \leq R$,
		\begin{equation*}
			I_{\substack{\theta\delta \leq \n{x-w} \leq R\\\n{x-w}\leq \n{y-w}}}(k) \leq 4\pi \Nrm{\rho}{L^\infty} \int_{\theta\delta}^R r^{-1} \d w = 4\pi \Nrm{\rho}{L^\infty} \ln_+\!\(\frac{R}{\theta\delta}\).
		\end{equation*}
		With the same estimates symmetrized in $(x,y)$, one finally gets for any $(\theta,\vartheta)\in(0,1)^2$,
		\begin{align*}
			I &:= I_{\Rd}(1) + I_{\Rd}(2)
			\\
			&\leq 2\frac{\Nrm{\rho}{L^1}}{R^3} + 8\pi \Nrm{\rho}{L^\infty} \(\frac{\theta}{1-\theta} + \frac{\vartheta^2}{2\(1-\vartheta\)^2} + \ln_+\!\(\frac{R}{\theta\delta}\) + \ln_+\!\(\frac{R}{\vartheta\delta}\)\).
		\end{align*}
		For $k\in\set{1,2}$, let $\theta_k\in(0,1)$ be the solution of $\theta_k^k = \(1-\theta_k\)^{k+1}$. Then the optimal $\theta$ and $\vartheta$ above are of the form $\theta = \min(\theta_k,\frac{R}{\delta})$. To simplify, we just take $\theta = \theta_1$ and $\vartheta = \theta_2$. This gives
		\begin{equation*}
			I \leq 2\frac{\Nrm{\rho}{L^1}}{R^3} + 8\pi \Nrm{\rho}{L^\infty} \(\frac{3}{2}-\theta_1-\frac{1}{2}\theta_2 + \ln_+\!\(\frac{R}{\theta_1\delta}\) + \ln_+\!\(\frac{R}{\theta_2\delta}\)\)
		\end{equation*}
		where $\theta_1 = \frac{3-\sqrt{5}}{2}\in(0.38,0.39)$ and $\theta_2 \in (0.43,0.44)$.
		Optimizing in $R$ when $\delta/R$ is small suggests to take $R^3 = \frac{3\,\Nrm{\rho}{L^1}}{8\pi \Nrm{\rho}{L^\infty}}$, yielding
		\begin{align*}
			I &\leq 8\pi \Nrm{\rho}{L^\infty} \(\frac{2}{3} + \frac{3}{2}-\theta_1-\frac{1}{2}\theta_2 + \ln_+\!\(\frac{R}{\theta_1\delta}\) + \ln_+\!\(\frac{R}{\theta_2\delta}\)\)
			\\
			&< 8\pi \Nrm{\rho}{L^\infty} \(\frac{2}{3} +\frac{\sqrt{5}-\theta_2}{2} + \ln_+\!\(\frac{R^2}{\theta_1^2\delta^2}\)\).
		\end{align*}
	\end{proof}


\bibliographystyle{abbrv} 
\bibliography{Vlasov}

\end{document}